\documentclass[12pt,reqno]{amsart}

\pdfoutput=1

\usepackage[letterpaper,margin=1.2in]{geometry}
\usepackage[backend=biber,style=numeric,isbn=false,url=false]{biblatex}
\usepackage{setspace}

\bibliography{extremehecke}

\newtheorem{thm}{Theorem}

\newtheorem{lem}{Lemma}

\newcommand{\C}{\mathbb{C}}
\newcommand{\R}{\mathbb{R}}
\newcommand{\Q}{\mathbb{Q}}

\newcommand{\OO}{\mathcal{O}}

\newcommand{\Real}{\mathrm{Re} \,}
\newcommand{\Imag}{\mathrm{Im} \,}
\newcommand{\sgn}{\mathrm{sgn}}
\newcommand{\dx}{\mathrm{d}}
\newcommand{\mf}{\mathfrak}

\usepackage{changepage}
\usepackage{amssymb}
\usepackage{amsmath}
\usepackage{graphicx}

\usepackage{xcolor}
\definecolor{couleur_cite}{rgb}{0.05,.4,0.05}
\definecolor{couleur_link}{rgb}{0.05,0.05,0.4}
\definecolor{couleur_url}{rgb}{0.5,0,0}
\usepackage[hyperfootnotes=false]{hyperref}
\hypersetup{hypertexnames=false,colorlinks=true,citecolor=couleur_cite,linkcolor=couleur_link,urlcolor=couleur_url,
pdfstartview=FitH, pdfauthor=Daniel White, pdftitle=Extreme values of Hecke L-functions}

\begin{document}

\title[Extreme values of Hecke $L$-functions]{Extreme values of Hecke $L$-functions to angular characters}
\author{Daniel White}
\address{Gettysburg College, Department of Computer Science, 300 North Washington Street, Gettysburg, PA 17325}
\email{dwhite@gettysburg.edu}

\begin{abstract}
Let $K$ be an imaginary quadratic number field and let $L(s,\xi_{\ell})$ denote the Hecke $L$-function to an angular character $\xi_{\ell}$ with frequency $\ell$. We detect values of $\log |L(\tfrac12,\xi_{\ell})|$ with size at least $(\sqrt{2} + o_K(1))(\log X / \log \log X)^{1/2}$ along each dyadic range $X \leqslant \ell \leqslant 2X$. This result relies on the resonance method, which is applied for the first time to this family of $L$-functions, where the classification and extraction of diagonal terms depends on the geometry of the complex embedding of $K$.
\end{abstract}

\subjclass[2020]{Primary 11R42; Secondary 11R04, 11R11}
\keywords{$L$-functions, extreme values, angular characters, Hecke characters, imaginary quadratic number fields}

\maketitle
\thispagestyle{empty}

\section{Introduction}
\label{extreme_values_intro}
The resonance method developed by Soundararajan \cite{Sound2008} is a flexible tool for detecting both unusually large and unusually small values in families of $L$-functions. In its original application to the zeta function, this method exploits the oscillatory behavior of $\zeta(\tfrac12+it)$ by generating a Dirichlet polynomial
\begin{equation*}
R(t) = \sum_{n \leqslant N} r(n) n^{-it}
\end{equation*}
that ``resonates'' with the ``louder frequencies'' of $\zeta(\tfrac12 + it)$ on $T \leqslant t \leqslant 2T$ for some later optimized choice of real-valued resonator coefficients $r(n)$. The argument proceeds by unpacking, reassembling, and analyzing what is essentially the first moment of
\begin{equation}
\label{resonated}
\zeta( \tfrac12 + it ) |R(t)|^2 \approx \sum_{k \leqslant T} k^{-1/2} \sum_{n,m \leqslant N} r(n) r(m) \left( \frac{n}{mk} \right)^{it}
\end{equation}
on this interval, where one must balance the length $N$ of the resonator --- a parameter commensurate with the strength of the resulting bound \eqref{sound_zeta} --- against the contribution of off-diagonal terms ($n \neq mk$) which are difficult to manage. Comparing the mean value of \eqref{resonated} to the mean of $|R(t)|^2$ and optimizing the resonator coefficients yields
\begin{equation}
\label{sound_zeta}
\max_{T \leqslant t \leqslant 2T} \log \left |\zeta ( \tfrac12 + it ) \right| \geqslant (1 + o(1)) \sqrt{\frac{\log T}{\log \log T}}.
\end{equation}
The bound \eqref{sound_zeta} substantially improved upon the previously best known result \cite{Bala1986,BalaRama1977} and is much larger than the bulk of the normal distribution with mean value $0$ and variance $(\log \log T)/2$; see \cite[\S 11.13]{Titch1986}. This result is fairly close to predictions from random matrix theory \cite{FGH2007} that suggest the right hand side of \eqref{sound_zeta} may be within a factor of $(\log \log T)/\sqrt{2}$ to the true maximum.

Bondarenko and Seip \cite{BondSeip2018,BondSeip2017} have since managed to substantially extend the length of the resonator for the zeta function, thus improving \eqref{sound_zeta}, but at the cost of localization; the full strength of their result requires the entire segment $0 \leqslant t \leqslant T$ and can not recover information along dyadic ranges.

For a number field $K$, the pure adelic analogue for the set of values of $\zeta(\tfrac12 + it)$ is the set of central values of Hecke $L$-functions to Hecke characters that ramify only at a single place; see the upcoming Section \ref{Hecke_characters} for a brief recollection of Hecke characters. Over $\Q$ this presents the additional problem of detecting large values of $L(\tfrac12, \chi)$ for Dirichlet characters $\chi$ over prime power moduli. In a related but not identical direction, Hough \cite{Hough2016} used the resonance method to establish a result of quality similar to $\eqref{sound_zeta}$ within prescribed angular sectors for Dirichlet characters varying on prime moduli, while de la Bret\`{e}che and Tenenbaum \cite{BretTenen2019} have since improved the quality of this bound for such even characters (without the angular sector restriction) to the level of Bondarenko and Seip's result on the zeta function.

The complexity of Hecke characters expands as one moves to proper extensions of $\Q$. For a fundamental example beyond both the archimedean and Dirichlet character types, when $K$ is an imaginary quadratic number field, \emph{angular characters} emerge from sole ramification at the unique complex place. On principal ideals of the ring of integers $\OO_K$, these characters behave precisely as $\xi_{\ell}((\beta)) = e(\ell \arg \beta / 2 \pi)$ for integers $\ell$ divisible by $|\OO_K^{\times}|$, where $\beta$ is an element of a fixed complex embedding of $\OO_K$. The integer $\ell$ is said to be the \emph{frequency} of the character and there are multiple angular characters of any given permissible frequency when $\OO_K$ is not a PID; this multiplicity is exactly equal to the size of the class group. The purpose of this article is to apply for the first time the resonance method to Hecke $L$-functions to angular characters to address for the first time the problem of detecting their large values, complementing the progress made on the archimedean and Dirichlet character types over $\Q$. We prove the following.

\begin{thm}
\label{main_result}
Let $K$ be an imaginary quadratic number field and $\mf{F}_{\ell}$ be its associated set of angular characters with frequency $\ell$. Then
\[
\max_{\substack{X \leqslant \ell \leqslant 2X \\ \xi_{\ell} \in \mf{F}_{\ell}}} \log \left| L( \tfrac{1}{2}, \xi_{\ell} ) \right| \geqslant \left( \sqrt{2} + o_K(1) \right) \sqrt{\frac{\log X}{\log \log X}}.
\]
\end{thm}
The most substantial contextual difference in our application of the resonance method arises in Sections \ref{offdiagonal_section} and \ref{diagonal_contribution} when one seeks to isolate and then analyze the diagonal terms (the analogues of $n=mk$ in \eqref{resonated}) of the analogous quantity which appears, very roughly, in this article as
\begin{equation}
\nonumber
\sum_{\substack{ \mf{N}\mf{a},\mf{N}\mf{b} \leqslant N \\ \mf{N} \mf{k} \ll X \\  \mf{k} \mf{a}\overline{\mf{b}} \text{ principal}}} \frac{r(\mf{a}) r(\mf{b})}{\sqrt{\mf{N} \mf{k}}} \sum_{k \asymp X} e( k\arg \gamma_{\mf{k} \mf{a} \overline{\mf{b}}} / 2 \pi),
\end{equation}
where $\gamma_{\mf{n}}$ is the generator of the principal ideal $\mf{n}$ nearest to the positive real line with respect to rotation about the origin. A similar, but more relaxed, analysis precedes this in Section \ref{offdiagonal_section} when examining the mean of our resonator polynomial. At this point, one must consider the geometry of the embedding of $K$ into $\mathbb{C}$. Indeed, the diagonal condition becomes characterized by whether the product $\mf{k} \mf{a} \overline{\mf{b}}$ is principally generated by a rational element of $\OO_K$. This is the only way the innermost sum above can avoid meaningful cancellation when $N$ is bounded appropriately in terms of $X$. 

We therefore bound length of the resonator in terms of $X$ (as is typical) so that our analysis of the quantity illustrated above depends on an area $\C$ where the embedding of $\OO_K$ is ``sufficiently discretized'' such that all $\gamma_{\mf{n}}$ \emph{near the real line} are \emph{exactly rational}. As a consequence, the off-diagonal terms where the $\mf{k} \mf{a} \overline{\mf{b}}$ are not rationally generated will not meaningfully contribute.

\subsection{Notation} In this article, unless stated otherwise, we let $\varepsilon$ represent a small, fixed positive number that may differ from line to line and let $K$ be an imaginary quadratic number field. As usual, $f \ll g$ and $f = O(g)$ mean that $|f| \leqslant Cg$ for some constant $C>0$ which may differ from line to line and depends on no parameters except for possibly $\varepsilon$, the number field $K$, or any other parameters explicitly stated. Additionally, $f=o(g)$ indicates that $f/g \to 0$ in a limit that will be clear from context. As is customary in analytic number theory, we write $e(z)$ to mean $\exp(2 \pi i z)$.

\subsection{Acknowledgments} The author would like to thank Djordje Mili\'cevi\'c for his thoughtful support and our many helpful discussions during the writing of this article. The author would like to further recognize and thank Bryn Mawr College for its support during the writing of this article, which was completed in partial fulfillment of the requirements for the degree of Doctor of Philosophy in the Department of Mathematics.

\section{Preliminaries}
\subsection{Hecke characters}
\label{Hecke_characters}
We now recall the basic properties of Hecke characters; the reader may reference \cite[Section~7.6]{Neukirch} for additional details. Let $K$ be a general number field and $\OO_K$ its ring of integers. We will denote the group of non-zero fractional ideals as $\mf{I}$; it contains the subgroup $\mf{P}$ of non-zero fractional principal ideals. The quotient group $\mf{I}/\mf{P}$ is the class group, which has finite size $h_K$. For a non-zero integral ideal $\mf{m} \subseteq \OO_K$, we further denote $\mf{I}_{\mf{m}}$ as the set of all fractional ideals coprime to $\mf{m}$.

Suppose $K$ has $r_1$ real embeddings and $r_2$ complex embeddings up to conjugation. A Hecke character modulo $\mf{m}$ is a continuous homomorphism $\xi: \mf{I}_{\mf{m}} \to S^1$ (when interpreted as a morphism on the idele class group of $K$) for which there are two characters
\[
\chi_f: (\OO_K/\mf{m})^{\times} \to S^1 \quad \text{and} \quad \chi_{\infty}: (\R^{\times})^{r_1} \times (\C^{\times})^{r_2} \to S^1
\]
such that $\xi((\beta)) = \chi_f(\beta) \chi_{\infty}(\beta)$ for every $\beta \in \OO_K$ where $((\beta),\mf{m})=1$. A pair of characters $(\chi_f,\chi_{\infty})$ is associated to a Hecke character if and only if $\chi_f(\epsilon) \chi_{\infty}(\epsilon) = 1$ for every $\epsilon \in \OO_K^{\times}$. The Hecke characters associated to such a pair differ by multiplication of a character of the class group. Let $\mathcal{H}$ denote the set of all Hecke characters of $K$.

We now focus our attention to when $K$ is an imaginary quadratic number field and fix an embedding into $\C$ so that no ambiguity may occur from conjugation. Hence we may use $\beta \in \OO_K$ to mean both the ring element and its image under the embedding. Let $\omega_K = |\OO_K^{\times}|$ and for each integer $\ell$ divisible by $\omega_K$ define the set
\begin{equation}
\label{characters_definition}
\mf{F}_{\ell} := \left\{ \xi_{\ell} \in \mathcal{H}: \xi_{\ell}\big((\beta) \big) = e \left(\frac{\ell \arg \beta}{2 \pi} \right) \text{ for all non-zero } \beta \in \OO_K \right\}
\end{equation}
of angular characters with frequency $\ell$ as in Theorem \ref{main_result}, where $|\mf{F}_{\ell}| = h_K$. In the event that $\omega_K \nmid \ell$, we simply take $\mf{F}_{\ell}$ to be empty. It will be useful to have a standard choice of generator for principal ideals that we encounter. For a principal ideal $\mf{n}$, let $\gamma_{\mf{n}}$ denote the generator of $\mf{n}$ whose principal argument resides in $[-\pi \omega_K^{-1},\pi \omega_K^{-1})$.

\subsection{Functional equation}
For any number field $K$, we may associate to any Hecke character its Hecke $L$-function
\[
L(s, \xi) = \sum_{\mf{k} \neq 0} \xi(\mf{k}) \mf{N} \mf{k}^{-s}
\]
which converges absolutely for $\Real s >1$ and has meromorphic continuation to $\C$. Each such $L$-function possesses a functional equation. In particular, for the $L$-functions of interest in this article, ~\cite[Theorem~3.8]{IwanKowal2004} boils down to the following.

\begin{thm}
\label{functional_equation_thm}
Let $D$ denote the discriminant of the imaginary quadratic number field $K$ and $\xi_{\ell}$ be an angular character of frequency $\ell$. The function
\begin{equation}
\label{lambda_definition}
\Lambda(s,\xi_{\ell}) = |D|^{s/2} (2 \pi)^{-s} \Gamma(s + |\ell|/2) L(s,\xi_{\ell})
\end{equation}
is entire for $\xi_{\ell}$ non-trivial and satisfies the functional equation
\begin{equation}
\label{functional_equation}
\Lambda(s,\xi_{\ell}) = \Lambda(1-s,\overline{\xi_{\ell}}).
\end{equation}
\end{thm}

When examining $L$-functions along the critical line, it is often beneficial to express the $L$-function as a smoothly weighted Dirichlet series in what is known as the approximate functional equation. Roughly speaking, and in general, this weight has approximate value one until its transition point, where it then drops rapidly to zero. One may think of the approximate functional equation as essentially a finite sum. For Hecke $L$-functions, the location of this point and the decay rate are both subject to the complexity of the associated character.

\begin{thm}
\label{AFE}
For $K$ and $\xi_{\ell}$ as in Theorem \ref{functional_equation_thm},
\[
L\left(\tfrac12,\xi_{\ell}\right) = \sum_{\mf{k}\neq (0)} \frac{\xi_{\ell}(\mf{k}) + \overline{\xi_{\ell}}(\mf{k})}{\sqrt{\mf{N}\mf{k}}} W_K \left(\mf{N}\mf{k}, |\ell| \right)
\]
where
\begin{equation}
\label{cutoff_definition}
W_K\bigg(\frac{|D|^{1/2}y}{2 \pi}, 2x \bigg) = V(y,x) = \frac{1}{2 \pi i} \int_{(1)} \rho_s(x) y^{-s} \frac{ \dx s}{s}
\end{equation}
and
\begin{equation}
\label{rho_definition}
\rho_s(x) = \frac{\Gamma (x+s+\tfrac12)}{\Gamma (x + \tfrac12 )}.
\end{equation}
\end{thm}

\begin{proof}
Let $\Lambda(s,\xi_{\ell})$ be as in \eqref{lambda_definition} in Theorem \ref{functional_equation_thm}. By a contour shift permitted by the exponential decay of the gamma function along vertical strips, we find
\[
\frac{1}{2 \pi i} \int_{(1)} \Lambda (s+\tfrac12,\xi_{\ell} ) \frac{\dx s}{s} = \Lambda (\tfrac12,\xi_{\ell} ) + \frac{1}{2 \pi i} \int_{(-1)} \Lambda (s + \tfrac12,\xi_{\ell} ) \frac{\dx s}{s}.
\]
On the right side above, applying the substitution $s \mapsto -s$ and the functional equation \eqref{functional_equation} from Theorem \ref{functional_equation_thm} before substituting \eqref{lambda_definition} and integrating term-wise on each side, we obtain
\[
L(\tfrac12,\xi_{\ell}) =  \frac{1}{2 \pi i} \sum_{\mf{k} \neq 0} \frac{\xi_{\ell}(\mf{k}) + \overline{\xi_{\ell}}(\mf{k})}{\sqrt{\mf{N}\mf{k}}} \int_{(1)} \rho_s(|\ell|/2) \left( 2 \pi |D|^{-1/2} \mf{N} \mf{k} \right)^{-s} \frac{\dx s}{s},
\]
as desired.
\end{proof}

The following lemmata characterize the behavior of $\rho_s(x)$ and its derivatives so that we may determine where $V(y,x)$, as defined in \eqref{cutoff_definition}, is ``essentially supported'' and how quickly it decays. The first of these demonstrates that, in suitable bounded vertical strips, $\rho_{\sigma + it}(x)$ has an approximately linear phase in $t$.

\begin{lem}
\label{rho_oscillatory}
Let $x\geqslant 1$ and $s = \sigma + it$ where $|\sigma| \leqslant x/2$. For $\rho_s(x)$ as in \eqref{rho_definition}, we have
\begin{align*}
\rho_s(x) = (x/e)^{\sigma} \left( 1 + \sigma/x \right)^{x + \sigma} \exp \left(F_{x+\sigma}(t) +  i \psi(t) \right) \left( 1 + O\left(x^{-1} \right) \right)
\end{align*}
with
\[
F_{\kappa}(t) = -\int_0^t \arctan \left( u/ \kappa \right) \dx u
\]
and real-valued phase
\[
\psi(t) = t \log \left( x + \sigma \right) + O \left(|t|^3 x^{-2} \right).
\]
\end{lem}

\begin{proof}
We begin by recalling Stirling's formula \cite[Theorem 2.3]{SteinShakarchi}, which states that 
\begin{equation*}
\Gamma(s) = \exp (s \log s - s) \left( 2 \pi / s \right)^{1/2} \left( 1 + O \left( |s|^{-1} \right) \right)
\end{equation*}
as $|s| \to \infty$ within any sector of $\C$ which omits the negative real axis. Applying Stirling's formula to \eqref{rho_definition} and noting \eqref{expand_log} shows
\begin{align}
\nonumber
\rho_s(x) & = \frac{(x + s + \tfrac12 )^{x + s}}{(x+ \tfrac12 )^x} e^{-s} \left( 1 + O(x^{-1}) \right) \\
& \label{rho_stirlings} = (x/e)^{s} \left( 1+s/x \right)^{x + s} \left( 1 + O\left( x^{-1} \right) \right)
\end{align}
where the expansion of $\log(1 + z)$ revealed
\begin{align}
    \nonumber (x + s+\tfrac12)^{x+s} & = (x+s)^{x+s} \left( 1 + \tfrac12(x+s)^{-1} \right)^{x+s} \\
    & \label{expand_log} = \sqrt{e} (x+s)^{x+s} \left(1 + O\left(x^{-1}\right) \right),
\end{align}
and one notes that the error term above is uniformly bounded for $|\Real s | \leqslant x/2$. Toward explicating the magnitude and phase of \eqref{rho_stirlings}, we may factor $(1+s/x)^{x+s}$ as the product of
\begin{align}
\nonumber
(1+ s/x)^{x+\sigma} & = (1 + \sigma/x)^{x+\sigma} \left(1 + \tfrac{it}{x+\sigma} \right)^{x+\sigma}\\
& \nonumber = (1+\sigma/x)^{x+\sigma} \exp \left( \tfrac12 (x+\sigma) \log \big( 1 + \left(\tfrac{t}{x + \sigma}\right)^2 \big) + i \phi_1(t) \right),\\
\nonumber & \text{where} \\
\nonumber \phi_1(t) & = (x + \sigma) \arctan \left( \tfrac{t}{x+\sigma} \right) \\
& \label{rho_exp_1} = t + O \left( |t|^3x^{-2} \right),
\end{align}
and
\begin{align}
\nonumber
(1+s/x)^{it} & = (1 + \sigma/x)^{it} \left(1 + \tfrac{it}{x+\sigma} \right)^{it}\\
& \nonumber = \exp \left( -t \arctan \left( \tfrac{t}{x + \sigma} \right) + i \phi_2(t) \right),\\
& \nonumber \text{where} \\
\nonumber \phi_2(t) & = t \log( 1 + \sigma/x ) + \tfrac{t}{2} \log \big( 1 + \left( \tfrac{t}{x + \sigma} \right)^2 \big) \\
& \label{rho_exp_2} = t \log (1 + \sigma/x) + O (|t|^3 x^{-2}),
\end{align}
noting that both $\phi_1$ and $\phi_2$ are real-valued. Inserting each of \eqref{rho_exp_1} and \eqref{rho_exp_2} back into \eqref{rho_stirlings} shows
\begin{align*}
\rho_s(x) = & \left( 1 + O\left( x^{-1} \right) \right) (x/e)^{\sigma} \left( 1 + \sigma/x \right)^{x + \sigma} \\
& \times \exp \left( \tfrac12 (x+\sigma) \log \big( 1 + \left(\tfrac{t}{x + \sigma}\right)^2 \big) -t \arctan \left( \tfrac{t}{x + \sigma} \right) \right)  \\
& \times \exp \left( i \phi_1(t) + i \phi_2(t) + it(-1 + \log x) \right) 
\end{align*}
thus yielding
\[
\rho_s(x) = (x/e)^{\sigma} \left( 1 + \sigma/x \right)^{x + \sigma} \exp \left(\tilde{F}_{x+\sigma}(t) +  i \psi(t) \right) \left( 1 + O\left(x^{-1} \right) \right)
\]
for $\psi(t)$ as defined in the statement of Lemma \ref{rho_oscillatory} and
\begin{equation}
\label{F_tilde}
\tilde{F}_{\kappa}(t) := - t \arctan \left( \tfrac{t}{\kappa} \right) + \tfrac12 \kappa \log \left( 1 + t^2/\kappa^2 \right).
\end{equation}
Note that $\tilde{F}_{\kappa}(0) = 0$ and $\tilde{F}_{\kappa}'(t) = - \arctan(t/\kappa)$, so that an application of the fundamental theorem of calculus shows $F_{\kappa}(t) = \tilde{F}_{\kappa}(t)$, finishing the proof.
\end{proof}

\begin{lem}
\label{logarithmic_derivative_Gamma}
For $n \geqslant 1$ and $\Real z \geqslant 1$, we have
\[
\frac{\dx^n}{\dx z^n} \log \Gamma(z) = \delta_n^1 \log z + O_n\left(\Real(z)^{-\delta_{n}^1-n+1}\right),
\]
where $\delta_n^m$ is the Kronecker delta function.
\end{lem}

\begin{proof}
We begin with the well-known identity \cite[\S 12.31]{WhittakerWatson}
\begin{equation*}
\frac{\dx}{\dx z} \log \Gamma(z) = \log z - \frac{1}{2z} -  \int_0^{\infty} \left( \frac{1}{2} - \frac{1}{u} + \frac{1}{e^u - 1} \right) e^{-uz} \, \dx u.
\end{equation*}
The case $n=1$ follows from the exponential decay of the integrand above, which begins at $\Real(z)^{-1}$. For $n \geqslant 2$, see that
\[
\frac{\dx^n}{\dx z^n} \log \Gamma (z) = (-1)^n (n-1)! \left( \frac{1}{n z^{n-1}} + \frac{1}{2z^n} + O\left( \Real(z)^{-n} \right) \right),
\]
which finishes the proof.
\end{proof}

\begin{lem}
\label{rho_derivatives_lemma}
Let $x \geqslant 1$ and $s=\sigma + it$ with $|\sigma| \leqslant x/2$. Then, for $n \geqslant 0$,
\begin{equation}
\label{rho_derivatives}
\rho_s^{(n)}(x) = \rho_s(x) \log^n (1 + s/x)\left( 1 + O_n(x^{-1})\right).
\end{equation}
\end{lem}

\begin{proof}
The case $n=0$ is clear and the case $n=1$ follows from the logarithmic derivative
\[
\frac{\dx}{\dx x} \rho_s(x) = \rho_s(x) \frac{\dx}{\dx x} \log \rho_s(x)
\]
and Lemma \ref{logarithmic_derivative_Gamma}. Proceeding inductively for $n \geqslant 2$, note that
\begin{align*}
    \rho_s^{(n)}(x) & = \frac{\dx^{n-1}}{\dx x^{n-1}} \left( \rho_s(x) \frac{\dx}{\dx x} \log \rho_s(x) \right)\\
    & = \sum_{m=0}^{n-1} \binom{n-1}{m} \rho_s^{(m)}(x) \frac{\dx^{n-m}}{\dx x^{n-m}} \log \rho_s(x).
\end{align*}
By our inductive hypothesis \eqref{rho_derivatives} for $m < n$ and Lemma \ref{logarithmic_derivative_Gamma}, we note the main contribution from $m = n-1$ and conclude the proof.
\end{proof}

Having established Lemmata \ref{rho_oscillatory} and \ref{rho_derivatives_lemma}, we are now equipped to understand the general behavior of $V(y,x)$ and its derivatives in the approximate functional equation of $L(\tfrac12,\xi_{\ell})$ as introduced in Theorem \ref{AFE}.

\begin{lem}
\label{cutoff_lem}
Let $y \geqslant 1$, $x \geqslant 3$, $n \geqslant 0$, and $V(y,x)$ be as in \eqref{cutoff_definition}. Then
\begin{equation}
\label{cutoff_bounds}
\frac{\partial^n}{\partial x^n} V(y,x) = \delta_n^0 + O_n \Big( x^{-n/2} (y/x)^{\sqrt{x}} \Big) \quad \text{and} \quad \frac{\partial^n}{\partial x^n} V(y,x) \ll_n x^{-n/2} (x/y)^{\sqrt{x}},
\end{equation}
where $\delta_n^m$ is the Kronecker delta function.
\end{lem}

\begin{proof}
Let $|\sigma| = x^{1/2}$. Shifting the line of integration in \eqref{cutoff_definition} to $(\sigma)$ (perhaps picking up contribution from the pole at $s=0$), differentiating $n$ times with respect to $x$, and substituting $s = \sigma + it$ yields
\begin{equation}
\label{cutoff_contour_shift}
\frac{\partial^n}{\partial x^n} V(y,x) = \delta_n^0 \Delta_{\sigma} + \frac{y^{-\sigma}}{2 \pi \sigma} \int_{-\infty}^{\infty} \rho^{(n)}_{\sigma + it}(x) y^{-it} G(\sigma,t) \, \dx t,
\end{equation}
where $\Delta_{\sigma} \in \{0,1\}$ indicates whether $\sigma$ is negative and
\[
G(\sigma,t) = \frac{\sigma}{\sigma + it}.
\]
We will start by bounding the tails of the integral in \eqref{cutoff_contour_shift}, corresponding to the values $|t| \geqslant x^{1/2 + \varepsilon}$. Exercising Lemmata \ref{rho_oscillatory} and \ref{rho_derivatives_lemma} with $(1+\sigma/x)^{x + \sigma} = \exp(\sigma + O(1))$ from our choice $|\sigma| = x^{1/2}$ shows these tails are bounded by an absolute constant multiple of
\begin{align*}
\int_{|t| \geqslant x^{1/2 + \varepsilon}} \left|  \rho_{\sigma + it}^{(n)}(x) \right| \dx t & \ll_n \int_{|t| \geqslant x^{1/2 + \varepsilon}} \left|  \rho_{\sigma + it}(x) \log^n \left(1 +  \tfrac{\sigma + it}{x}\right)\right| \dx t \\
& \ll x^{\sigma - n} \int_{x^{1/2+\varepsilon}}^{\infty} \exp \left( n \log t - \int_0^{t} \arctan \left( \tfrac{u}{x + \sigma} \right) \ \dx u \right) \dx t.
\end{align*}
Noting that $\arctan(v) \geqslant v \pi/4$ for $0 \leqslant v \leqslant 1$ and $\arctan(v) \geqslant \pi/4$ for $v \geqslant 1$ allows us to bound the right hand side above by
\begin{align*}
 \leqslant x^{\sigma-n} & \int_{x^{1/2 + \varepsilon}}^{x + \sigma} \exp \left( \tfrac{-t^2 \pi }{8x}(1+o_n(1)) \right) \dx t \\
 & + x^{\sigma-n} \exp \left( \tfrac{\pi}{8} (x + \sigma) \right) \int_{x + \sigma}^{\infty} \exp \left( -\tfrac{\pi}{4} t (1+o_n(1)) \right) \dx t 
\end{align*}
where utilizing the bound $\int_c^{\infty} e^{-u^2} \dx u \ll e^{-c^2}/c$ then shows that the main contribution above comes from the left summand whose integral is bounded by
\[
\ll_n \int_{x^{1/2 + \varepsilon}}^{\infty} \exp \left( \frac{-t^2 \pi}{16x} \right) \, \dx t \ll x^{1/2 - \varepsilon} \exp  ( - x^{\varepsilon}  \pi/16 ).
\]
We therefore find that the integral in \eqref{cutoff_contour_shift} becomes
\begin{equation}
\label{cutoff_main}
O_n\left(x^{\sigma - n + 1/2 - \varepsilon} \exp(-x^{\varepsilon} \pi/ 16) \right) + \int_{-x^{1/2 + \varepsilon}}^{x^{1/2 + \varepsilon}} \rho^{(n)}_{\sigma + it}(x) y^{-it} G(\sigma,t) \dx t.
\end{equation}
Another application of lemmata \ref{rho_oscillatory} and \ref{rho_derivatives_lemma}, keeping in mind $(1+\sigma/x)^{x + \sigma}$ is equal to $\exp(\sigma + O(1))$ by our choice $|\sigma| = x^{1/2}$, shows that the integral in \eqref{cutoff_main} is
\begin{equation}
\label{center_integral}
\ll_n x^{\sigma} \int_{-x^{1/2+\varepsilon}}^{x^{1/2+\varepsilon}} \exp \left( F_{x + \sigma}(t) \right) \left| \log^n \left(1 +  \tfrac{\sigma + it}{x}\right) \right| \, \dx t
\end{equation}
where, by \eqref{F_tilde}, we may expand
\begin{align}
\nonumber F_{x + \sigma}(t) &= -t \arctan \left( \tfrac{t}{x+\sigma} \right) + \tfrac12 (x+ \sigma) \log \left( 1 + \tfrac{t^2}{(x + \sigma)^2} \right)\\
& \label{F_expand} = -\frac{t^2}{2 x}\left( 1 + O(x^{-1/2} + t^2 x^{-2}) \right).
\end{align}
Inserting \eqref{F_expand} into \eqref{center_integral} yields a value
\[
\ll x^{\sigma} \int_{-x^{1/2 + \varepsilon}}^{x^{1/2 + \varepsilon}} \exp(-\tfrac14 t^2/x) \left| \log^n \left(1 +  \tfrac{\sigma + it}{x}\right) \right| \, \dx t \ll x^{\sigma - n + 1/2} \sigma^{n} 
\]
which can be used as an upper bound for the integral in \eqref{cutoff_main} resulting, ultimately, in a bound for \eqref{cutoff_contour_shift} of the form
\begin{align*}
\frac{\partial^n}{\partial x^n} V(y,x) & = \delta_n^0 \Delta_{\sigma} + O_n \left( \frac{y^{-\sigma}}{2 \pi \sigma} \left(  x^{\sigma - n + 1/2} \sigma^{n} \right) \right)\\
& = \delta_n^0 \Delta_{\sigma} + O_n \left( x^{-n/2}(x/y)^{\sigma} \right)
\end{align*}
as desired.
\end{proof}

It will be convenient to produce lower bounds on sums of products involving factors of $W_K$ by truncating the sums, where all other factors of the summands are non-negative. To do so, we must prove that $V(y,x)$ is also non-negative for the input $(y,x)$ under consideration. Note that the bounds in \eqref{cutoff_bounds} are insufficient here. We therefore prove the following lemma which happens to also show that $V(y,x)$ can be expressed as the ratio of an incomplete Gamma function to a complete Gamma function.

\begin{lem}
\label{cutoff_positive}
If $y > 0$ and $x > 0$, then $V(y,x) \in (0,1)$.
\end{lem}

\begin{proof}
We begin by shifting the contour in the definition of $V(y,x)$ to the piece-wise contour $\gamma_{\varepsilon}^- + \gamma_{\varepsilon} + \gamma_{\varepsilon}^+$, where the $\gamma_{\varepsilon}^{\pm}$ sit along $\pm i[\varepsilon, \infty)$ and $\gamma_{\varepsilon}$ forms the right half of the circle with radius $\varepsilon$ centered at the origin. Define
\begin{equation*}
I_{\varepsilon}^{\pm} = \frac{1}{2 \pi i} \int_{\gamma_{\varepsilon}^{\pm}} \rho_s(x) y^{-s} \frac{\dx s}{s}
\end{equation*}
and $I_{\varepsilon}$ similarly for $\gamma_{\varepsilon}$, where we denote the limit as $\varepsilon \to 0^+$ of each $I_{\varepsilon}^{\pm}$ and $I_{\varepsilon}$ as $I^{\pm}$ and $I$, respectively. Note that $I = 1/2$ and
\begin{align*}
    I^- + I^+ & = \frac{1}{\pi \Gamma(x + \tfrac12)} \int_0^{\infty} \Imag \Gamma (x + \tfrac12 + it) y^{-it} \frac{\dx t}{t} \\
    &= \frac12 \Gamma(x+\tfrac12)^{-1} \int_0^{\infty} u^{x-1/2} e^{-u} \, \sgn \log (u/y)  \, \dx u
\end{align*}
which implies that $I^- + I^+ \in \R$ is strictly less than $1/2$ in absolute value. The statement of the lemma follows.
\end{proof}

\section{The setup}
\label{setup}
Fix an imaginary quadratic number field $K$ and let $X$ be a large positive value. We begin by introducing our notion of a resonator
\begin{equation}
\label{resonator_defn}
R(\xi_{\ell}) = \sum_{\mf{N} \mf{a} \leqslant N} \xi_{\ell}(\mf{a}) r(\mf{a})
\end{equation}
where $\xi_{\ell} \in \mf{F}_{\ell}$ is a Hecke character as defined in \eqref{characters_definition} and the $r(\mf{a})$ are non-negative resonator coefficients. The length of our resonator will be set by a constraint involving the frequencies of the $\xi_{\ell}$ under consideration, as is expected within the resonator method, where the goal is to build $R(\xi_{\ell})$ to resonate with the oscillatory function $L(\tfrac12 ,\xi_{\ell})$ for $\xi_{\ell}$ varying among the $\mf{F}_{\ell}$ with $X \leqslant \ell \leqslant 2X$. One then expects large central values to be pronounced when examining $L(\tfrac12,\xi_{\ell})|R(\xi_{\ell})|^2$. To make this idea precise, let $0 \leqslant \Phi(x) \leqslant 1$ be a smooth weight function supported in $[1,2]$ with uniformly bounded derivatives and note that
\begin{equation}
\label{ratio}
\max_{\substack{X \leqslant \ell \leqslant 2X \\ \xi_{\ell} \in \mf{F}_{\ell}}} \left| L(\tfrac12, \xi_{\ell}) \right| \geqslant \frac{\big| \sum_{\ell} \Phi(\ell/X) \sum_{\xi_{\ell} \in \mf{F}_{\ell}} L(\tfrac12, \xi_{\ell}) |R(\xi_{\ell})|^2  \big|}{\sum_{\ell} \Phi(\ell/X) \sum_{\xi_{\ell} \in \mf{F}_{\ell}} |R(\xi_{\ell})|^2}.
\end{equation}

We make provisions for the resonator coefficients before an explicit choice is later made. Let $r \geqslant 0$ be multiplicative, $r(\mf{p})=0$ on prime ideals $\mf{p}$ that are not generated by a factor of a split rational prime, and for all prime ideals $\mf{p}$, $r(\mf{p}) = r(\overline{\mf{p}})$ and $r(\mf{p}^m) = 0$ for $m \geqslant 2$. These choices do not affect the quality of our result and serve to simplify many calculations that are encountered.

\subsection{Bounding the denominator}
\label{bounding_denominator}
To start, we examine the denominator of the ratio in \eqref{ratio}, upon which we wish to produce an upper bound. By orthogonality of the class group characters and the definition of $\gamma_{\mf{n}}$, both discussed in Section \ref{Hecke_characters}, we have
\begin{align}
    \nonumber \sum_{\omega_K \mid \ell} \Phi(\ell/X) \sum_{\xi_{\ell} \in \mf{F}_{\ell}} |R(\xi_{\ell})|^2 & = \sum_{\omega_K \mid \ell} \Phi(\ell/X) \sum_{\mf{N}\mf{a},\mf{N}\mf{b} \leqslant N} r(\mf{a}) r(\mf{b}) \sum_{\xi_{\ell} \in \mf{F}_{\ell}} \xi_{\ell}(\mf{a}\overline{\mf{b}})\\
    & \label{denominator_swaporder} = h_K \sum_{\substack{ \mf{N}\mf{a},\mf{N}\mf{b} \leqslant N \\ \mf{a}\overline{\mf{b}} \in P}} r(\mf{a}) r(\mf{b}) \sum_{k} \Phi(\omega_K k/X) e( \omega_K k\arg \gamma_{\mf{a} \overline{\mf{b}}} / 2 \pi),
\end{align}
where $P$ denotes the set of principal ideals of $\OO_K$. Application of Poisson summation with a change of variable shows that the above equals
\begin{align}
\label{denominator_poisson}
h_K \omega_K^{-1} X \sum_{\substack{\mf{N}\mf{a},\mf{N}\mf{b} \leqslant N \\ \mf{a}\overline{\mf{b}} \in P}} r(\mf{a}) r(\mf{b}) \sum_j \hat{\Phi} \left(X (\omega_K^{-1} j - \arg \gamma_{\mf{a} \overline{\mf{b}}} / 2\pi) \right).
\end{align}
The rapid decay of $\hat{\Phi}$ allows us to profitably split this sum into two cases according to whether $\left| \arg \gamma_{\mf{a}\overline{\mf{b}}} \right|$ passes the threshold $X^{-1 + \varepsilon}$. Note the series expansion of $\sin \arg \gamma_{\mf{a}\overline{\mf{b}}}$ in the case $\left| \arg \gamma_{\mf{a} \overline{\mf{b}}} \right| < X^{-1+\varepsilon}$ shows that
\begin{align}
\label{imaginary_bound}
\left|\Imag \gamma_{\mf{a} \overline{\mf{b}}} \right| \leqslant N\left(X^{-1+\varepsilon} + O\left(X^{-2}\right) \right).
\end{align}
It is here that a constraint on the length of the resonator appears should we wish to avoid off-diagonal contribution in \eqref{denominator_poisson}; later, though, we further restrict the length. Let $P_0$ denote the set of principal ideals generated by rational numbers and $P'$ be the set of ideals which contain no $P_0$-factors. Enforcing $N \leqslant X^{1 - \varepsilon}$ guarantees $\left|\Imag \gamma_{\mf{a} \overline{\mf{b}}} \right| < 1$ by \eqref{imaginary_bound}, so that $\mf{a} \overline{\mf{b}} \in P_0$ for $X$ sufficiently large relative to $\varepsilon$.

By considering the prime factorization of $\mf{a}$ and $\mf{b}$, we may write
\begin{equation}
\label{factor_ideals}
    \mf{a} = \mf{a}_0 \mf{a}' \quad \text{and} \quad \mf{b} = \mf{b}_0 \mf{b}'
\end{equation}
uniquely for $\mf{a}_0, \mf{b}_0 \in P_0$ and $\mf{a'}, \mf{b}' \in P'$ to find that $\mf{a}\overline{\mf{b}} \in P_0$ if and only if $\mf{a}' = \mf{b}'$. By the bound $\hat{\Phi}(y) \ll_n y^{-n}$ (derived from repeated integration by parts) and the deduction above, the total contribution to the sum in \eqref{denominator_poisson} over all such pairs $(\mf{a},\mf{b})$ is
\begin{align}
    \nonumber \sum_{\substack{\mf{N}\mf{a}, \mf{N}\mf{b} \leqslant N \\ \mf{a}\overline{\mf{b}} \in P_0}} r(\mf{a})r(\mf{b}) \big(\hat{\Phi}(0) + O(X^{-1})\big)
    & \ll \hat{\Phi}(0) \sum_{\mf{a}_0,\mf{b}_0 \in P_0} r(\mf{a}_0)r(\mf{b}_0) \sum_{\substack{\mf{a}' \in P' \\ (\mf{a}',\mf{a}_0\mf{b}_0) = 1}} r(\mf{a}')^2 \\
    & \label{denominator_diagonal} = \hat{\Phi}(0) \sideset{}{'} \prod_{\mf{p}} \left( 1 + 4r(\mf{p})^2 + r(\mf{p})^4 \right),
\end{align}
where the primed product above is indexed over arbitrary representatives of conjugate pairs of factors of split rational primes. Again wielding the above bound on $\hat{\Phi}(y)$ and exercising the Cauchy--Schwarz inequality, we find that the off-diagonal contribution (\emph{i.e.} when $\left| \arg \gamma_{\mf{a}\overline{\mf{b}}} \right| \geqslant X^{-1 + \varepsilon}$) to the sum in \eqref{denominator_poisson} is
\begin{equation}
\label{denominator_off-diagonal}
    \ll_{m} X^{-\varepsilon m} N \sum_{\mf{a} \neq 0} r(\mf{a})^2 \leqslant X^{1-\varepsilon m} \prod_{\mf{p}} (1 + r(\mf{p})^2).
\end{equation}
Choosing $m$ sufficiently large relative to $\varepsilon$ allows the quantity in \eqref{denominator_diagonal} to dominate that in \eqref{denominator_off-diagonal}. Combining this observation with \eqref{denominator_swaporder} and \eqref{denominator_poisson} yields
\begin{equation}
\label{denominator}
    \sum_{\ell} \Phi(\ell/X) \sum_{\xi_{\ell} \in \mf{F}_{\ell}} |R(\xi_{\ell})|^2 \ll X \hat{\Phi}(0) \sideset{}{'} \prod_{\mf{p}} \left( 1 + 4r(\mf{p})^2 + r(\mf{p})^4 \right).
\end{equation}

\section{Resonator application}
\label{resonator_application}
Application of Theorem \ref{AFE} to the numerator in \eqref{ratio} produces
\begin{align}
\label{numerator_AFE}
    \sum_{\ell} \Phi(\ell/X) \sum_{\xi_{\ell} \in \mf{F}_{\ell}}  |R(\xi_{\ell})|^2 \sum_{\mf{k}} \frac{\xi_{\ell}(\mf{k}) + \overline{\xi_{\ell}}(\mf{k})}{\sqrt{\mf{N}\mf{k}}} W_K \big( \mf{N} \mf{k}, |\ell| \big)
\end{align}
upon which we wish to place a lower bound. Our analysis of the sum over $\xi_{\ell}(\mf{k})$ above will end up having built-in symmetry with the dual sum over $\overline{\xi_{\ell}}(\mf{k})$. Further, the main contribution to the sum over $\xi_{\ell}(\mf{k})$, hence also $\overline{\xi_{\ell}}(\mf{k})$, will be real-valued. We thus explicitly develop our argument first (and only) for the sum in \eqref{numerator_AFE} over $\xi_{\ell}(\mf{k})$. With this in mind, we proceed as in \eqref{denominator_swaporder} and \eqref{denominator_poisson} -- exchanging order of summation and using Poisson summation -- to obtain
\begin{align}
    & \nonumber h_K \sum_{\substack{\mf{N} \mf{a}, \mf{N} \mf{b} \leqslant N \\ \mf{k} \mf{a} \overline{\mf{b}} \in P }} \frac{r(\mf{a}) r(\mf{b})}{\sqrt{\mf{N} \mf{k}}} \sum_{k} \Phi(\omega_K k /X) W_K(\mf{N} \mf{k}, \omega_K k) e\left( \omega_K k \arg \gamma_{\mf{a} \overline{\mf{b}} \mf{k}} / 2 \pi \right)\\
    & \label{numerator_poisson} = h_K \omega_K^{-1} X \sum_{\substack{\mf{N} \mf{a}, \mf{N} \mf{b} \leqslant N \\ \mf{k} \mf{a} \overline{\mf{b}} \in P }} \frac{r(\mf{a}) r(\mf{b})}{\sqrt{\mf{N} \mf{k}}} \sum_{j} \mathcal{F}_u [\Phi(u) W_K(\mf{N} \mf{k}, u X)]\left( X \left( \omega_K^{-1}j - \arg \gamma_{\mf{a} \overline{\mf{b}} \mf{k}} / 2 \pi \right) \right)
\end{align}
where $\mathcal{F}_u[f]$ denotes the Fourier transform of $f(u)$. Over the course of this section, we will establish a lower bound on \eqref{numerator_AFE} by extracting the contribution to \eqref{numerator_poisson} from the case
\begin{equation*}
    \mf{k}\mf{a}\overline{\mf{b}} \in P_0 \quad  \text{and} \quad \mf{N}\mf{k} \ll X.
\end{equation*}
For $\mf{k}\mf{a}\overline{\mf{b}} \in P \setminus P_0$, the rapid decay of the Fourier transform in \eqref{numerator_poisson} will be quite useful in our management, provided that $\arg \gamma_{\mf{k}\mf{a}\overline{\mf{b}}}$ is guaranteed to not be too small. Once focused on $\mf{k}\mf{a}\overline{\mf{b}} \in P_0$ (and its $j=0$ term in \eqref{numerator_poisson}), we will be permitted to truncate at $\mf{N}\mf{k} \ll X$ by the positivity of $\Phi(u) W_K(\mf{N} \mf{k}, uX)$ which is established in Lemma \ref{cutoff_positive}. Afterward, the resulting sum will undergo a somewhat lengthy factorization process enabled by Rankin's trick, as recalled in Section \ref{factorization_section}.

We begin by forming a handle on the decay of the transform of the function
\[\Phi(u)W_K(\mf{N}\mf{k},uX)
\]
above. It will suffice to place bounds on its derivatives. Note that
\begin{align}
\nonumber
\frac{\dx^n}{\dx u^n} \Phi(u) W_K(\mf{N} \mf{k}, uX) = \sum_{0 \leqslant m \leqslant n} \binom{n}{m} \Phi^{(n-m)}(u) \frac{\dx^m}{\dx u^m} W_K (\mf{N} \mf{k},uX)
\end{align}
and recall $\Phi$ is supported in $[1,2]$ with uniformly bounded derivatives, as stated in Section \ref{setup}. Setting $c_K = 4 \pi / \sqrt{D}$ and using \eqref{cutoff_bounds} from Lemma \ref{cutoff_lem}, we find that the above is
\begin{equation}
\label{derivatives} \ll_n \left\{
\begin{array}{ll}
X^{n/2} \left( \frac{u X}{c_K \mf{N}\mf{k}} \right)^{\sqrt{uX/2}}, & \text{if } u \leqslant c_K \mf{N}\mf{k}/X; \\
1 + X^{n/2} \left( \frac{u X}{c_K \mf{N}\mf{k}}  \right)^{-\sqrt{uX/2}}, & \text{if } u > c_K \mf{N} \mf{k}/X.
\end{array}
\right.
\end{equation}
Repeating integration by parts $n$ times on the Fourier transform of interest and exercising the bound in \eqref{derivatives} shows that, for $v \neq 0$,
\begin{align}
\label{transform_bound_1}
\mathcal{F}_u [\Phi (u) W_K (\mf{N} \mf{k};uX)] (v) \ll_n v^{-n} X^{n/2} S(\mf{N}\mf{k},X) + v^{-n} \max \left( 0,2-\tfrac{c_K \mf{N}\mf{k}}{X} \right)
\end{align}
where
\begin{equation}
\label{S_definition} S(\mf{N}\mf{k},X) := \int_{\substack{u \leqslant c_K \mf{N} \mf{k}/X \\ u \in (1,2)}} \left( \tfrac{uX}{c_K \mf{N} \mf{k}} \right)^{\sqrt{uX/2}} \dx u + \int_{\substack{u \geqslant c_K \mf{N} \mf{k} / X \\ u \in (1,2)}} \Big( \tfrac{c_K \mf{N}\mf{k}}{uX} \Big)^{\sqrt{uX/2}} \dx u.
\end{equation}
Using the substitution $w = uX/(c_K\mf{N}\mf{k})$ in \eqref{S_definition} and bounding the exponents trivially based on the interval of integration reveals
\begin{align}
\label{S_bound}
S(\mf{N}\mf{k},X) \ll_n
\begin{cases}
X^{-1/2} \left(c_K \mf{N} \mf{k}/X \right)^{\sqrt{X/2}}, & \mf{N} \mf{k} \leqslant X/c_K; \\
X^{-1/2}, & X/c_K < \mf{N} \mf{k} \leqslant 2X/c_K; \\
X^{-1/2} \left(\tfrac12 c_K \mf{N} \mf{k} /X \right)^{-\sqrt{X/2}}, & 2X/c_K < \mf{N} \mf{k},
\end{cases}
\end{align}
so that the insertion of \eqref{S_bound} into \eqref{transform_bound_1}, combined with a replacement of $n$ with any real-valued $\eta \geqslant 1$ via interpolation (a geometric mean argument), yields
\begin{align}
& \nonumber \mathcal{F}_u [\Phi(u) W_K(\mf{N}\mf{k};uX)](v) \\
& \label{transform_bound_2} \ll v^{-\eta} \chi_{(0,2X/c_K]}(\mf{N}\mf{k}) + \frac{X^{\eta/2-1/2}}{v^{\eta}} \min \left( 1,\left(c_K \mf{N} \mf{k}/X \right)^{\sqrt{X/2}},\left(\tfrac12 c_K \mf{N} \mf{k} /X \right)^{-\sqrt{X/2}} \right).
\end{align}
The next two subsections will break the outer sum of \eqref{numerator_poisson} into two cases: the \emph{diagonal case} $\mf{k}\mf{a}\overline{\mf{b}} \in P_0$ and the \emph{off-diagonal case} $\mf{k}\mf{a}\overline{\mf{b}} \in P \setminus P_0$.

\subsection{Off-diagonal contribution}
\label{offdiagonal_section}
We first consider the \emph{off-diagonal case} $\mf{k}\mf{a}\overline{\mf{b}} \in P \setminus P_0$, and we will do so with four subcases, the first of which is the \emph{off-diagonal subcase I}: $\mf{N}\mf{k} \leqslant 2X/c_K$ and $|\arg \gamma_{\mf{k}\mf{a}\overline{\mf{b}}}| \geqslant X^{-1/2}$. It will be beneficial to subdivide the sum over $\mf{k}$ in \eqref{numerator_poisson} according to the size of $\mf{N}\mf{k}$ along the dyadic partitions
\[
D_{\nu}(X) = (2^{-\nu}X/c_K, 2^{-\nu + 1}X/c_K]
\]
and $\arg \gamma_{\mf{k} \mf{a} \overline{\mf{b}}}$ according to the intervals
\[
I_m(X) = [mX^{-1/2}, (m+1)X^{-1/2}).
\]
By the bound in \eqref{transform_bound_2}, the contribution to \eqref{numerator_poisson} for $\mf{N} \mf{k} \leqslant 2X/c_K$ with $\eta \geqslant 1$ is
\begin{equation}
\label{apply_transform_bound}
    \ll X^{1/2 - \eta / 2} \sum_{\mf{N} \mf{a}, \mf{N}\mf{b} \leqslant N} r(\mf{a}) r(\mf{b}) \sum_{\substack{ \mf{k}\mf{a}\overline{\mf{b}} \in P \\ \mf{N}\mf{k} \leqslant 2X/c_K}} \mf{N} \mf{k}^{-1/2} \sum_j \left( \omega_K^{-1} j - \arg \gamma_{\mf{k}\mf{a}\overline{\mf{b}}} / 2 \pi \right)^{- \eta}
\end{equation}
where, by the definition of $\gamma_{\mf{n}}$, the main contribution to the sum over $j$ above comes from $j=0$. Precisely,
\begin{equation}
\label{poisson_inner_bound}
\sum_j \left( \omega_K^{-1} j - \arg \gamma_{\mf{k}\mf{a}\overline{\mf{b}}} / 2 \pi \right)^{- \eta} \ll_{\eta} (\arg \gamma_{\mf{k}\mf{a}\overline{\mf{b}}})^{-\eta}.
\end{equation}
By the restriction on the argument of $\gamma_{\mf{k}\mf{a}\overline{\mf{b}}}$ in this subcase, we now bound the contribution to \eqref{apply_transform_bound} as
\begin{align}
\nonumber & \ll_{\eta} X^{1/2 - \eta/2} \sum_{\mf{N}\mf{a},\mf{N}\mf{b} \leqslant N} r(\mf{a}) r(\mf{b}) X^{\eta/2} \sum_{m \geqslant 1} m^{-\eta} \sum_{0 \leqslant \nu \leqslant \log_2 X} 2^{\nu/2} X^{-1/2} \sum_{\substack{\mf{N} \mf{k} \in D_{\nu}(X) \\ | \arg \gamma_{\mf{k}\mf{a}\overline{\mf{b}}}| \in I_m(X) \\ \mf{k}\mf{a}\overline{\mf{b}} \in P \setminus P_0}} 1 \\
& \nonumber \ll_{\eta} \sum_{\mf{N}\mf{a},\mf{N}\mf{b} \leqslant N} r(\mf{a}) r(\mf{b}) \sum_{0 \leqslant \nu \leqslant \log_2 X} 2^{\nu/2} \sum_{\substack{\mf{N} \mf{k} \in D_{\nu}(X) \\ | \arg \gamma_{\mf{k}\mf{a}\overline{\mf{b}}}| \in I_m(X) \\ \mf{k}\mf{a}\overline{\mf{b}} \in P \setminus P_0}} 1
\end{align}
Bounding the inner sum above can take place by counting the number of $\OO_K$ lattice points associated to each $\mf{k} \mf{a} \overline{\mf{b}}$ (translating $\mf{k}$ to a generator in $\C$ by one of a fixed \emph{finite} number of representatives of the class group) within the annulus defined by $D_{\nu}(X)$ while restricted to the appropriate angular sector. Note that the area of this annular sector is commensurate to $2^{- \nu} X^{1/2}$ and the perimeter to $2^{-\nu/2} X^{1/2}$, allowing us to estimate the number of lattice points as $\ll 2^{ -\nu /2} X^{1/2} + 1$.

Therefore the total contribution to \eqref{numerator_poisson} from \emph{off-diagonal subcase I} is
\begin{equation}
\label{offdiagonal_subcase_1_bound} \ll_{\eta} X^{1/2} \log X \sum_{\mf{N}\mf{a},\mf{N}\mf{b} \leqslant N} r(\mf{a}) r(\mf{b}).
\end{equation}
We next consider the \emph{off-diagonal subcase II}: $\mf{N}\mf{k} \leqslant 2X/c_K$ and $|\arg \gamma_{\mf{k}\mf{a}\overline{\mf{b}}}| < X^{-1/2}$. By the series expansion of sine around zero, we have $|\arg \gamma_{\mf{k}\mf{a}\overline{\mf{b}}}| \gg N^{-1} \mf{N}\mf{k}^{-1/2}$ for $\mf{k}\mf{a} \overline{\mf{b}} \in P \setminus P_0$, and so using \eqref{transform_bound_2} to again obtain \eqref{apply_transform_bound} with its associated bound \eqref{poisson_inner_bound} on the sum over $j$ shows that the contribution to \eqref{numerator_poisson} from the \emph{off-diagonal subcase II} is
\begin{align}
& \nonumber \ll_{\eta} X^{1/2-\eta/2} \sum_{\mf{N}\mf{a},\mf{N}\mf{a} \leqslant N} r(\mf{a}) r(\mf{b}) \sum_{0 \leqslant \nu \leqslant \log_2 X} (X 2^{-\nu})^{-1/2} \sum_{\substack{\mf{N}\mf{k} \in D_{\nu}(X) \\ \mf{k}\mf{a}\overline{\mf{b}} \in P \setminus P_0 \\ |\arg \gamma_{\mf{k}\mf{a}\overline{\mf{b}}}| < X^{-1/2} }} N^{\eta} (X 2^{-\nu})^{\eta/2} \\
\nonumber & \ll N^{\eta} \sum_{\mf{N}\mf{a},\mf{N}\mf{a} \leqslant N} r(\mf{a}) r(\mf{b}) \sum_{0 \leqslant \nu \leqslant \log_2 X} 2^{\nu/2 - \nu \eta / 2} \sum_{\substack{\mf{N}\mf{k} \in D_{\nu}(X) \\ \mf{k}\mf{a}\overline{\mf{b}} \in P \setminus P_0 \\ |\arg \gamma_{\mf{k}\mf{a}\overline{\mf{b}}}| < X^{-1/2} }} 1.
\end{align}
Bounding the inner sum above can take place by counting lattice points within the annulus defined by $D_{\nu}(X)$ while restricted to the appropriate angular sector, as done in the previous subcase. This again allows us to estimate the number of lattice points as $\ll 2^{ -\nu /2} X^{1/2} + 1$. Substitution to achieve an upper bound then shows the contribution to \eqref{numerator_poisson} from \emph{off-diagonal subcase II} is
\begin{equation}
\label{offdiagonal_subcase_2_bound}
\ll_{\eta} X^{1/2}N^{\eta} \sum_{\mf{N} \mf{a}, \mf{N} \mf{b} \leqslant N} r(\mf{a})r(\mf{b}).
\end{equation}

In the remaining two cases, we will split $\mf{N}\mf{k}$ according to the intervals
\[
D'_{\nu}(X) = (\nu 2X/c_K,(\nu+1)2X/c_K].
\]
Now consider the \emph{off-diagonal subcase III}: $\mf{N} \mf{k} > 2X/c_K$ and $|\arg \gamma_{\mf{k}\mf{a}\overline{\mf{b}}}| \geqslant X^{-1/2}$. With respect to \eqref{transform_bound_2}, notice that we have entered the range where
\[
\mathcal{F}_u [\Phi(u) W_K(\mf{N}\mf{k};uX)](v)
\]
decays rapidly. Again exercising \eqref{transform_bound_2}, one finds that the contribution to \eqref{numerator_poisson} with the constraint $\mf{N}\mf{k} > 2X/c_K$ and  observation \eqref{poisson_inner_bound} is
\begin{equation}
\label{apply_transform_bound3}
    \ll_{\eta} X^{1/2 - \eta/2} \left( \frac{2X}{c_K}  \right)^{\sqrt{X/2}} \sum_{\mf{N}\mf{a},\mf{N}\mf{b} \leqslant N} r(\mf{a}) r(\mf{b}) \sum_{\substack{ \mf{k}\mf{a}\overline{\mf{b}} \in P \setminus P_0 \\ \mf{N}\mf{k} > 2X/c_K}} \mf{N} \mf{k}^{-1/2 - \sqrt{X/2}} (\arg \gamma_{\mf{k} \mf{a} \overline{\mf{b}}})^{-\eta}.
\end{equation}
We may bound the inner sum above
\begin{align}
\nonumber \sum_{\nu \geqslant 1} & \sum_{m \geqslant 1} \sum_{\substack{\mf{N}\mf{k} \in D_{\nu}'(X) \\ \mf{k} \mf{a} \overline{\mf{b}} \in P\\ | \arg \gamma_{\mf{k}\mf{a}\overline{\mf{b}}}| \in I_m(X)}} \mf{N} \mf{k}^{-1/2-\sqrt{X/2}} (\arg \gamma_{\mf{k} \mf{a} \overline{\mf{b}}})^{-\eta} \\
& \label{inner_bound3} \ll X^{\eta/2} \left( \frac{2X}{c_K}  \right)^{-1/2 - \sqrt{X/2}} \sum_{\nu \geqslant 1} \nu^{-1/2 - \sqrt{X/2}} \sum_{m \geqslant 1} m^{-\eta} \sum_{\substack{ \mf{N}\mf{k} \in D_{\nu}'(X) \\ \mf{k} \mf{a} \overline{\mf{b}} \in P \\ |\arg \gamma_{\mf{k} \mf{a} \overline{\mf{b}}}| \in I_m(X)}} 1
\end{align}
where a lattice point counting argument, explained previously in \emph{off-diagonal subcase I}, can be applied where the area of the angular sector is of size $X^{1/2}$ and the perimeter of size $\nu^{-1/2} X^{1/2}$. We therefore bound the inner sum of \eqref{inner_bound3} as $\ll X^{1/2}$. Subsequent evaluation and substitution of \eqref{inner_bound3} into \eqref{apply_transform_bound3} shows that the contribution from \emph{off-diagonal subcase III} to \eqref{numerator_poisson} is
\begin{equation}
    \label{offdiagonal_subcase_3_bound} \ll_{\eta} X^{1/2} \sum_{\mf{N} \mf{a}, \mf{N} \mf{b} \leqslant N} r(\mf{a})r(\mf{b}).
\end{equation}

Finally, there is the \emph{off-diagonal subcase IV}: $\mf{N} \mf{k} \geqslant 2X/c_K$ and $|\arg \gamma_{\mf{k}\mf{a}\overline{\mf{b}}}| < X^{-1/2}$. With \eqref{apply_transform_bound3}, the bound $|\arg \gamma_{\mf{k} \mf{a} \overline{\mf{b}}}| \gg \mf{N}\mf{k}^{-1/2} N^{-1}$, and the same lattice point counting argument in previous cases, the contribution to \eqref{numerator_poisson} from the inner sum of \eqref{apply_transform_bound3} keeping in mind \eqref{poisson_inner_bound} is
\begin{align}
    \ll_{\eta} N^{\eta} \sum_{\nu \geqslant 1} \sum_{\substack{\mf{N} \mf{k} \in D_{\nu}' \\ |\arg \gamma_{\mf{k}\mf{a}\overline{\mf{b}}} | < X^{-1/2}}} \mf{N} \mf{k}^{-1/2 - \sqrt{X/2} + \eta/2} \ll N^{\eta} X^{\eta/2} \left( \frac{2X}{c_K} \right)^{-\sqrt{X/2}}.
\end{align}
Substitution of the above into \eqref{apply_transform_bound3} then bounds the contribution to \eqref{numerator_poisson} from \emph{off-diagonal subcase IV} by
\begin{equation}
\label{offdiagonal_subcase_4_bound}
\ll_{\eta} X^{1/2} N^{\eta} \sum_{\mf{N} \mf{a}, \mf{N} \mf{b} \leqslant N} r(\mf{a}) r(\mf{b}).
\end{equation}

In consideration of the bounds from each of \emph{off-diagonal subcases I-IV}, it is shown that the total contribution to \eqref{numerator_poisson} from the \emph{off-diagonal subcases I-IV} is bounded by \eqref{offdiagonal_subcase_4_bound}. Concisely, for $\eta \geqslant 1$ and $N \gg \log X$ (as will be chosen), we have
\begin{align}
\nonumber X \sum_{\substack{\mf{N} \mf{a}, \mf{N} \mf{b} \leqslant N \\ \mf{k} \mf{a} \overline{\mf{b}} \in P\setminus P_0 }} \frac{r(\mf{a}) r(\mf{b})}{\sqrt{\mf{N} \mf{k}}} \sum_{j} & \mathcal{F}_u [\Phi(u) W_K(\mf{N} \mf{k}, u X)] \left( X \left( \omega_K^{-1}j - \arg \gamma_{\mf{a} \overline{\mf{b}} \mf{k}} / 2 \pi \right) \right)\\
\label{offdiagonal_total}
& \ll_{\eta} X^{1/2} N^{\eta} \sum_{\mf{N}\mf{a},\mf{N}\mf{b} \leqslant N} r(\mf{a}) r(\mf{b}).
\end{align}

\subsection{Diagonal contribution}
\label{diagonal_contribution}
With respect to the contribution from $\mf{k}\mf{a}\overline{\mf{b}} \in P_0$ to \eqref{numerator_poisson}, one may readily handle the case $j \neq 0$ of the inner sum by the approach from the previous subsection since $\omega_K^{-1}j - \arg \gamma_{\mf{a} \overline{\mf{b}} \mf{k}} / 2 \pi$ here is at least as large as its counterpart in the off-diagonal case with $j=0$. This certainly yields no more than the quantity \eqref{offdiagonal_total} found in the off-diagonal case. Further, for $j=0$ in the inner sum of \eqref{numerator_poisson}, the non-negativity of $\Phi(u)W_K(\mf{N}\mf{k},uX)$ following from Lemma \ref{cutoff_positive} allows one use the outer sum according to $\mf{N}\mf{k} \leqslant X/ 2 c_K$ to produce a lower bound on \eqref{numerator_poisson}, provided that the term resulting from truncation is $\gg$ than the quantity in \eqref{offdiagonal_total} with appropriate implied constant.

Note that \eqref{cutoff_bounds} from Lemma \ref{cutoff_lem} demonstrates that $\mathcal{F}_u [\Phi(u) W_K(\mf{N} \mf{k}; uX)]\left( 0 \right)$ is effectively constant in this range. Write $\mf{a}=\mf{a}_0\mf{a}'$ and $\mf{b}=\mf{b}_0\mf{b}'$ as in \eqref{factor_ideals}. Further factorizing
\begin{equation}
\label{factor_ideals_2}
    \mf{a}'=(\mf{a}',\mf{b}')\mf{a}'' \quad \text{and} \quad \mf{b}' = (\mf{a}', \mf{b}')\mf{b}'',
\end{equation}
we note that $\mf{k}\mf{a}\overline{\mf{b}} \in P_0$ precisely when $\mf{k}=\mf{k}_0 \overline{\mf{a}''}\mf{b}''$ for $\mf{k}_0 \in P_0$. The truncated sum discussed above is of magnitude
\begin{align*}
X \sum_{\mf{N} \mf{a}, \mf{N} \mf{b} \leqslant N} r(\mf{a}) r(\mf{b}) \sum_{\substack{\mf{N}\mf{k} \leqslant X/2c_K \\ \mf{k} \mf{a} \overline{\mf{b}} \in P_0 }} \mf{N} \mf{k}^{-1/2} & = X\sum_{\mf{N} \mf{a}, \mf{N} \mf{b} \leqslant N} \frac{r(\mf{a}) r(\mf{b})}{\sqrt{\mf{N}(\mf{a}''\mf{b}'')}} \sum_{\mf{N}\mf{k}_0 \leqslant X/2 c_K\mf{N}(\mf{a}''\mf{b}'')} \mf{N} \mf{k}_0^{-1/2} \\
& \gg X N^{-1} \sum_{\mf{N} \mf{a}, \mf{N} \mf{b} \leqslant N} r(\mf{a}) r(\mf{b}).
\end{align*}
when $N \ll X^{1/2 + \varepsilon}$. Indeed, we now finally set $N = X^{1/4-\varepsilon}$ and $\eta = 1 + \varepsilon$ so that $X^{1/2}N^{\eta} = o(X/N)$. Keeping in mind the dual sum of \eqref{numerator_AFE}, we find that two times \eqref{numerator_poisson} --- thus ultimately \eqref{numerator_AFE} --- is bounded from below in absolute value by
\begin{equation}
\label{diagonal_extracted}
\left( 2h_K \omega_K^{-1} \hat{\Phi}(0) + o(1)\right) X \sum_{\substack{\mf{N} \mf{a}, \mf{N} \mf{b} \leqslant N \\ \mf{N}\mf{k} \leqslant X/2 c_K \\ \mf{k} \mf{a} \overline{\mf{b}} \in P_0 }} \frac{r(\mf{a}) r(\mf{b})}{\sqrt{\mf{N} \mf{k}}}.
\end{equation}

As done in Section \ref{bounding_denominator}, we wish to factorize \eqref{diagonal_extracted} so that we may more readily compare it to \eqref{denominator} on a subpower scale, thereby detecting large values of $|L(\tfrac12,\xi_{\ell})|$, which are expected to be subpower with respect to $X$; see again \eqref{ratio} and Theorem \ref{main_result}.

\subsection{Factorization of the diagonal sum}
\label{factorization_section}
As in \eqref{factor_ideals} and \eqref{factor_ideals_2}, write $\mf{a} = \mf{a}_0 \mf{a}'$ and $\mf{b} = \mf{b}_0 \mf{b}'$ so that
\begin{equation*}
    \mf{a} \overline{\mf{b}} = \mf{a}_0 \mf{b}_0 (\mf{a}',\mf{b}') \overline{(\mf{a}',\mf{b}')} \mf{a}'' \overline{\mf{b}''}
\end{equation*}
where $(\mf{a}'',\mf{b}'') = 1$ and $\mf{a}''\overline{\mf{b}''} \in P'$. Then for $\mf{ka} \overline{\mf{b}} \in P_0$, we must have
\begin{equation*}
    \mf{k} = \mf{k}_0 \overline{\mf{a}''} \mf{b}''
\end{equation*}
for $\mf{k}_0 \in P_0$. Moving forward, should context require it, we will assume in our notation that $\mf{a}$ and $\mf{b}$, and also $\mf{a}'$ and $\mf{b}'$, factorize as above in relation to one another. The sum in \eqref{diagonal_extracted} is thus
\begin{equation}
\label{diagonal_no_k}
\sum_{\mf{N} \mf{a}, \mf{N} \mf{b} \leqslant N} \frac{r(\mf{a}) r(\mf{b})}{\sqrt{\mf{N} (\mf{a}''\mf{b}'')}} \sum_{\mf{N} \mf{k}_0 \leqslant X/2c_k \mf{N} (\mf{a}'' \mf{b}'')} \mf{N} \mf{k}_0^{-1/2} \geqslant \sum_{\substack{\mf{N} \mf{a}, \mf{N} \mf{b} \leqslant N \\ (\mf{a}''\overline{\mf{a}''},\mf{b}''\overline{\mf{b}''}) = 1}} \frac{r(\mf{a}) r(\mf{b})}{\sqrt{\mf{N} (\mf{a}''\mf{b}'')}}.
\end{equation}
We now wish to complete the sum in this lower bound over all non-zero $\mf{a}$ and $\mf{b}$. We do exactly this using Rankin's trick, which is the observation that for any non-negative sequence $(m_n)$ and any $\alpha \geqslant 0$, we have
\begin{equation*}
\sum_{0 \leqslant n \leqslant N} m_n = \sum_{n \geqslant 0} m_n + O\left( N^{-\alpha} \sum_{n \geqslant 0} n^{\alpha} m_n \right).   
\end{equation*}
Using the above, first on the sum over $\mf{a}$ on the right hand side of \eqref{diagonal_no_k}, and then on the sum over $\mf{b}$ (both times using the same value of $\alpha \geqslant 0$ to be chosen later)  shows that the right side of \eqref{diagonal_no_k} equals
\begin{align}
\label{diagonal_complete}
\sum_{\substack{\mf{a}, \mf{b} \neq 0 \\ (\mf{a}''\overline{\mf{a}''},\mf{b}''\overline{\mf{b}''}) = 1}} \frac{r(\mf{a}) r(\mf{b})}{\sqrt{\mf{N} (\mf{a}''\mf{b}'')}} & + O \left( \Xi(\alpha, 0) + \Xi(\alpha, \alpha) \right)
\end{align}
where
\begin{equation}
    \Xi(\alpha_1, \alpha_2) = N^{-\alpha_1 - \alpha_2} \sum_{\substack{\mf{a}, \mf{b} \neq 0\\ (\mf{a}''\overline{\mf{a}''},\mf{b}''\overline{\mf{b}''}) = 1}} \frac{r(\mf{a}) r(\mf{b})}{\sqrt{\mf{N} (\mf{a}''\mf{b}'')}} \mf{N}\mf{a}^{\alpha_1} \mf{N}\mf{b}^{\alpha_2}.
\end{equation}

We first factorize the sum in $\Xi (\alpha, \alpha)$ of \eqref{diagonal_complete}, where the factorization of the expected main term in \eqref{diagonal_complete} will follow from the substitution $\alpha =0$. Afterward, we will factorize the sum $\Xi(\alpha, 0)$ in \eqref{diagonal_complete} in a similar way. To begin, write the sum in $\Xi (\alpha, \alpha)$ as
\begin{equation}
\label{diagonal_error}
 \sum_{\substack{\mf{a}', \mf{b}' \in P' \\ (\mf{a}''\overline{\mf{a}''},\mf{b}''\overline{\mf{b}''}) = 1}} \frac{r(\mf{a}') r(\mf{b}')}{\sqrt{\mf{N} (\mf{a}''\mf{b}'')}} \mf{N}(\mf{a}'\mf{b}')^{\alpha} \sum_{\substack{\mf{a}_0 \in P_0 \\ (\mf{a}_0,\mf{a}')=1}} r(\mf{a}_0)\mf{N}\mf{a}_0^{\alpha} \sum_{\substack{\mf{b}_0 \in P_0 \\ (\mf{b}_0,\mf{b}')=1}} r(\mf{b}_0)\mf{N}\mf{b}_0^{\alpha}.
\end{equation}
After factorizing the two inner sums of the above expression as
\begin{equation*}
\sideset{}{'} \prod_{\mf{p}} A_{\alpha}(\mf{p})^2 \prod_{\mf{p} \mid \mf{a}'} A_{\alpha}(\mf{p})^{-1} \prod_{\mf{p} \mid \mf{b}'} A_{\alpha}(\mf{p})^{-1}, \quad A_{\alpha}(\mf{p}) := (1 + r(\mf{p})^2 \mf{N} \mf{p}^{2\alpha}),
\end{equation*}
we write $\mf{c}' = (\mf{a}',\mf{b}')$ so that \eqref{diagonal_error} becomes
\begin{equation}
\label{diagonal_error2}
\sideset{}{'} \prod_{\mf{p}} A_{\alpha}(\mf{p})^2 \sum_{\mf{c}' \in P'} r(\mf{c}')^2 (\mf{N}\mf{c}')^{2\alpha} \prod_{\mf{p} \mid \mf{c}'} A_{\alpha}(\mf{p})^{-2} \sum_{\substack{(\mf{a}''\overline{\mf{a}''},\mf{b}''\overline{\mf{b}''}) = 1 \\ (\mf{a}''\mf{b}'',\mf{c}'\overline{\mf{c}'}) = 1}} \frac{r(\mf{a}'') r(\mf{b}'')}{(\mf{N}\mf{a}''\mf{b}'')^{1/2-\alpha}} \prod_{\mf{p} \mid \mf{a}''\mf{b}''} A_{\alpha}(\mf{p})^{-1}.
\end{equation}
With $\mf{c}'' = \mf{a}''\mf{b}''$ (recalling that $\mf{a}'', \mf{b}'' \in P'$ by their definitions) and consideration of the divisor function $d$, we may write the innermost sum above as
\begin{equation}
\label{diagonal_doubleprime}
\sum_{\substack{\mf{c}'' \in P' \\ (\mf{c}'',\mf{c}'\overline{\mf{c}'})=1}} \frac{r(\mf{c}'') d(\mf{c}'')}{(\mf{N}\mf{c}'')^{1/2 - \alpha}} \prod_{\mf{p} \mid \mf{c}''} A_{\alpha}(p)^{-1} = \sideset{}{'} \prod_{\mf{p} \nmid \mf{c}'\overline{\mf{c}'}} B_{\alpha}(\mf{p}), \quad B_{\alpha}(\mf{p}) := 1 + \frac{4r(\mf{p})}{\mf{N}\mf{p}^{1/2-\alpha}} A_{\alpha}(\mf{p})^{-1},
\end{equation}
so that inserting \eqref{diagonal_doubleprime} into the inner sum of \eqref{diagonal_error2} yields
\begin{align}
\nonumber
\sideset{}{'} \prod_{\mf{p}} & A_{\alpha}(\mf{p})^2 B_{\alpha}(\mf{p}) \sum_{\mf{c}' \in P'} r(\mf{c}')^2 (\mf{N} \mf{c}')^{2\alpha} \sideset{}{'} \prod_{\mf{p} \mid \mf{c}'\overline{\mf{c}'}} A_{\alpha}(\mf{p})^{-2} B_{\alpha}(\mf{p})^{-1} \\
\nonumber
& = \sideset{}{'} \prod_{\mf{p}} \left( A_{\alpha}(\mf{p})^2 B_{\alpha}(\mf{p}) + 2r(\mf{p})^2 \mf{N}\mf{p}^{2\alpha} \right)\\
\label{diagonal_error2_factorized}
&= \sideset{}{'} \prod_{\mf{p}} \left( 1 + \frac{4 r(\mf{p})}{\mf{N} \mf{p}^{1/2}} \mf{N}\mf{p}^{\alpha} + 4r(\mf{p})^2 \mf{N}\mf{p}^{2\alpha} + \frac{4 r(\mf{p})^3}{\mf{N} \mf{p}^{1/2}} \mf{N}\mf{p}^{3\alpha} + r(\mf{p})^4 \mf{N} \mf{p}^{4\alpha} \right)
\end{align}
as the factorization of the sum in $\Xi (\alpha, \alpha)$ in \eqref{diagonal_complete}.

Using a similar strategy, one may also factorize the sum in $\Xi (\alpha, 0)$ of \eqref{diagonal_complete}. This quantity is
\begin{equation}
\label{diagonal_error_1}
 \sum_{\substack{\mf{a}', \mf{b}' \in P' \\ (\mf{a}''\overline{\mf{a}''},\mf{b}''\overline{\mf{b}''}) = 1}} \frac{r(\mf{a}') r(\mf{b}')}{\sqrt{\mf{N} \mf{a}''\mf{b}''}} (\mf{N}\mf{a}')^{\alpha} \sum_{\substack{\mf{a}_0 \in P_0 \\ (\mf{a}_0,\mf{a}')=1}} r(\mf{a}_0)\mf{N}\mf{a}_0^{\alpha} \sum_{\substack{\mf{b}_0 \in P_0 \\ (\mf{b}_0,\mf{b}')=1}} r(\mf{b}_0).
\end{equation}
We may factorize the two inner sums of the above expression as
\begin{equation*}
\sideset{}{'} \prod_{\mf{p}} A_{\alpha}(\mf{p})A_0(\mf{p}) \prod_{\mf{p} \mid \mf{a}'} A_{\alpha}(\mf{p})^{-1} \prod_{\mf{p} \mid \mf{b}'} A_0(\mf{p})^{-1},
\end{equation*}
so that \eqref{diagonal_error_1} becomes
\begin{align}
\nonumber \sideset{}{'} \prod_{\mf{p}} A_{\alpha}(\mf{p})A_0(\mf{p}) & \sum_{\mf{c}' \in P'} r(\mf{c}')^2 (\mf{N}\mf{c}')^{\alpha} \prod_{\mf{p} \mid \mf{c}'} A_{\alpha}(\mf{p})^{-1} A_0(\mf{p})^{-1}  \\
& \label{intermediate_factorization} \times \sum_{\substack{(\mf{a}''\overline{\mf{a}''},\mf{b}''\overline{\mf{b}''}) = 1 \\ (\mf{a}''\mf{b}'',\mf{c}'\overline{\mf{c}'}) = 1}} \frac{r(\mf{a}'') r(\mf{b}'')}{(\mf{N}\mf{a}''\mf{b}'')^{1/2}} \prod_{\mf{p} \mid \mf{a}''} A_{\alpha}(\mf{p})^{-1} \mf{N}\mf{p}^{\alpha} \prod_{\mf{p} \mid \mf{b}''} A_0(\mf{p})^{-1}.
\end{align}
Rewriting the inner sum above as follows and factorizing allow us to obtain
\begin{align*}
    \sum_{\substack{\mf{c}'' \in P' \\ (\mf{c}'',\mf{c}'\overline{\mf{c}'})=1}} \frac{r(\mf{c}'')}{(\mf{N}\mf{c}'')^{1/2}} \prod_{\mf{p} \mid \mf{c}''} A_0(\mf{p})^{-1} & \sum_{\mf{a}'' \mid \mf{c}''} (\mf{N}\mf{a}'')^{\alpha} \prod_{\mf{p} \mid \mf{a}''} A_0(\mf{p}) A_{\alpha}(\mf{p})^{-1} \\
    & = \sideset{}{'} \prod_{\mf{p} \nmid \mf{c}'\overline{\mf{c}'}} \left( 1 + \frac{2 r(\mf{p})}{\mf{N}\mf{p}^{1/2}} (A_0(\mf{p})^{-1} + \mf{N} \mf{p}^{\alpha} A_{\alpha}(\mf{p})^{-1} ) \right)
\end{align*}
so that for
\begin{equation*}
    C_{\alpha}(\mf{p}) := A_{\alpha}(\mf{p}) A_0(\mf{p}) + \frac{2 r(\mf{p})}{\mf{N} \mf{p}^{1/2}}(A_{\alpha}(\mf{p}) + \mf{N} \mf{p}^{\alpha} A_0(\mf{p}))
\end{equation*}
we may further factorize \eqref{intermediate_factorization} as
\begin{equation*}
    \sideset{}{'} \prod_{\mf{p}} C_{\alpha}(\mf{p}) \sum_{\mf{c}' \in P'} r(\mf{c}')^2 (\mf{N}\mf{c}')^{\alpha} \prod_{\mf{p} \mid \mf{c}'} C_{\alpha}(\mf{p})^{-1} = \sideset{}{'} \prod_{\mf{p}} \left( C_{\alpha}(\mf{p}) + 2r(\mf{p})^2 \mf{N}\mf{p}^{\alpha} \right).
\end{equation*}
This ultimately yields for the sum in $\Xi (\alpha, 0)$ a factorization of
\begin{align}
\nonumber
\sideset{}{'} \prod_{\mf{p}} \bigg( 1 + \frac{2 r(\mf{p})}{\mf{N}\mf{p}^{1/2}} (1 + \mf{N}\mf{p}^{\alpha}) + r(\mf{p})^2 &  (1 + 2 \mf{N}\mf{p}^{\alpha} + \mf{N}\mf{p}^{2\alpha}) \\
& \label{diagonal_error1_factorized} + \frac{2r(\mf{p})^3}{\mf{N}\mf{p}^{1/2}}(\mf{N}\mf{p}^{\alpha} + \mf{N}\mf{p}^{2\alpha}) + r(\mf{p})^4 \mf{N} \mf{p}^{2\alpha} \bigg).
\end{align}

Once a choice of resonator coefficients is made, it will become necessary to show that
\begin{equation}
\label{target_bound_after_coefficients}
    \Xi(\alpha, 0) + \Xi(\alpha, \alpha) = o \left( \Xi(0,0) \right)
\end{equation}
so that the main term of \eqref{diagonal_complete} can be used as a lower bound on \eqref{diagonal_extracted}, and hence as a lower bound on the numerator of \eqref{numerator_AFE}.

\section{Choice of resonator coefficients}

We will now choose explicit values for all remaining parameters in our application of the resonance method. Recall the properties of $r(\mf{p})$ given in the beginning of Section \ref{setup} and let
\begin{equation}
\label{parameters}
    L = \sqrt{\tfrac12 \log N \log \log N}; \quad r(\mf{p}) = \frac{L}{\mf{N}\mf{p}^{1/2} \log \mf{N} \mf{p}}, \quad L^2 \leqslant \mf{N} \mf{p} \leqslant \exp (\log ^2 L),
\end{equation}
and $r(\mf{p}) = 0$ for prime ideals $\mf{p}$ outside of this range. Further, we set the Rankin's trick parameter $\alpha = 1/ \log^3 L$.

\subsection{Ratio of error terms to main term}

We will need to examine each of $\Xi(\alpha, 0)$, $\Xi(\alpha, \alpha)$, and $\Xi(0,0)$ on a logarithmic scale to demonstrate that the main term of \eqref{diagonal_complete} truly dominates the error terms. Expanding
\[
\mf{N} \mf{p}^{\alpha} = 1 + O (\alpha \log \mf{N} \mf{p}), \quad \mf{N} \mf{p} \leqslant \exp ( \log^2 L),
\]
then making the crude observations
\begin{align*}
 \sideset{}{'} \sum_{\mf{p}} & \frac{r(\mf{p})^3 \log \mf{N} \mf{p} }{\mf{N} \mf{p}^{1/2}} \ll L^3 \sum_{\mf{N} \mf{p} \geqslant L^2}  \mf{N}\mf{p}^{-2} (\log \mf{N} \mf{p})^{-2} \ll \frac{(\log N)^{1/2}}{ (\log \log N)^{5/2}}\\
& \text{and} \quad \sideset{}{'} \sum_{\mf{p}} r(\mf{p})^4 \log \mf{N} \mf{p} \ll L^4 \sum_{\mf{N} \mf{p} \geqslant L^2} \mf{N} \mf{p}^{-2} (\log \mf{N} \mf{p}^{-3}) \ll \frac{\log N}{(\log \log N)^3},
\end{align*}
lets us deduce from \eqref{diagonal_error2_factorized} and \eqref{diagonal_error1_factorized}, for $E_{\alpha} = \log N / (\log \log N)^3$, that
\begin{align}
\nonumber
& \frac{\Xi(\alpha, \alpha)}{\Xi(0,0)} + \frac{\Xi(\alpha, 0)}{\Xi(0,0)} \\
& \nonumber \ll \exp \left( -2 \alpha \log N + \sideset{}{'} \sum_{\mf{p}} \left( (\mf{N} \mf{p}^{\alpha} - 1)\frac{4 r(\mf{p})}{\sqrt{\mf{N}\mf{p}}} + (\mf{N}\mf{p}^{2\alpha} - 1) 4 r(\mf{p})^2 \right) \right)\\
& \nonumber \qquad \qquad \times \exp \left(  O(\alpha E_{\alpha})\right)\\
& \nonumber + \exp \left( -\alpha \log N + \sideset{}{'} \sum_{\mf{p}} \left( (\mf{N} \mf{p}^{\alpha} - 1)\frac{2 r(\mf{p})}{\sqrt{\mf{N}\mf{p}}} + (2\mf{N}\mf{p}^{\alpha} + \mf{N}\mf{p}^{2\alpha} - 3) r(\mf{p})^2 \right)\right)\\
& \label{rankin_ratio} \qquad \qquad \times \exp \left(  O(\alpha E_{\alpha})\right).
\end{align}
Consideration of the linear expansion
\[
\mf{N} \mf{p}^{\alpha} = 1 + \alpha \log \mf{N} \mf{p} + O((\alpha \log \mf{N} \mf{p})^2), \quad \mf{N} \mf{p} \leqslant \exp ( \log^2 L),
\]
in conjunction with the definitions
\begin{align}
   \nonumber G_{\alpha}(r) & := \sideset{}{'} \sum_{\mf{p}} \left( \alpha \frac{r(\mf{p}) \log \mf{N} \mf{p}}{\mf{N} \mf{p}^{1/2}}  + \alpha^2 r(\mf{p})^2 \log^2 \mf{N}\mf{p} \right), \\
   \label{rankin_final_sum} H_{\alpha}(r) & := -\alpha \log N + \alpha \sideset{}{'} \sum_{\mf{p}} 4 r(\mf{p})^2\log \mf{N}\mf{p},
\end{align}
allows us to rewrite \eqref{rankin_ratio} as
\begin{equation}
\label{rankin_bound}
\exp \big( 2H_{\alpha}(r) + O (G_{\alpha}(r)+ \alpha E_{\alpha }) \big) + \exp \big(H_{\alpha}(r) + O (G_{\alpha}(r) + \alpha E_{\alpha }) \big)
\end{equation}
where, by partial summation and the prime ideal theorem,
\begin{equation}
\label{G_bound} G_{\alpha}(r) \ll \alpha^2 \log N \log \log N \log \log \log N.
\end{equation}
Note the right hand side of the above dominates $\alpha E_{\alpha}$.

\subsection{Bounding the numerator}

Bounding sums similar to the one in \eqref{rankin_final_sum} is a general requirement of the resonance method and is delicately computed through partial summation and the prime ideal theorem. For example, this exact computation (after exercising the prime ideal theorem) is used in the \emph{proof} of \cite[Lemma 7.4]{SecondMoment} with the parameters $\omega(\mf{p}) = a_{\omega} = 2$ in their notation. Outsourcing this computation, we have
\begin{align*}
    H_{\alpha} (r) & \geqslant - \alpha \log N + \alpha \sum_{\mf{p}} \omega(p) r(\mf{p})^2 \log \mf{N} \mf{p} + O(1) \\
    & = - \alpha \frac{\log N \log \log \log N}{\log \log N} + O\left( \alpha \frac{\log N}{\log \log N} \right).
\end{align*}
Applying the bounds in \eqref{rankin_bound} and \eqref{G_bound}, we find that the ratio of the sum of errors to the expected main term in \eqref{diagonal_complete} is
\begin{equation*}
    \ll \exp \left( - \alpha \frac{\log N \log \log \log N}{2 \log \log N} \right).
\end{equation*}
Therefore, in consider of \eqref{diagonal_complete} and its relationship to \eqref{diagonal_extracted}, we have shown that the numerator of \eqref{ratio} is bounded from below by
\begin{equation}
\label{numerator}
    X h_K \omega_K^{-1} \hat{\Phi}(0) \sideset{}{'} \prod_{\mf{p}} \left( 1 + \frac{4 r(\mf{p})}{\mf{N} \mf{p}^{1/2}} + 4r(\mf{p})^2 + r(\mf{p})^4 \right).
\end{equation}

\section{Main result}

With respect to the ratio in \eqref{ratio}, we may now utilize our upper bound on its denominator \eqref{denominator} and our lower bound on its numerator \eqref{numerator}. Doing so yields
\begin{equation*}
\max_{\substack{X \leqslant \ell \leqslant 2X \\ \xi_{\ell} \in \mf{F}_{\ell}}} \log \left| L(\tfrac12, \xi_{\ell}) \right| \geqslant  o(1) + \sideset{}{'} \sum_{\mf{p}} \frac{4 r(\mf{p})}{\mf{N}\mf{p}^{1/2}} = (2 + o(1)) \frac{L}{\log L}.
\end{equation*}
Recalling our exact choice of parameters in \eqref{parameters} and the value $N=X^{1/4 - \varepsilon}$ set in Section \ref{diagonal_contribution} completes the proof of Theorem \ref{main_result}.

\printbibliography

\end{document}